\documentclass{article}

\usepackage{amsmath}
\usepackage{amssymb}
\usepackage{amsfonts}
\usepackage{latexsym}
\usepackage{amscd}
\usepackage[dvips]{graphicx}
\usepackage{amsthm}
\usepackage{enumerate}

\newcommand{\ul}{\underline}
\newcommand{\Ocal}{\mathcal{O}}
\newcommand{\Ccal}{\mathcal{C}}
\newcommand{\Zbb}{\mathbb{Z}}

    \newtheorem{Lem}{Lemma}[section]
    \newtheorem{Prop}[Lem]{Proposition}
    \newtheorem{Thm}[Lem]{Theorem}
    \newtheorem{Cor}[Lem]{Corollary}

\theoremstyle{definition}

\begin{document}

\title{Brill-Noether locus of rank 1\\ and degree $g-1$  on a nodal curve}
\author{Juliana Coelho and Eduardo Esteves}

\maketitle

\begin{abstract}
\noindent
In this paper we consider the Brill-Noether locus $W_{\ul d}(C)$ of line bundles of multidegree $\ul d$ of total degree $g-1$ having a nonzero section on a nodal reducible curve $C$ of genus $g\geq2$. 
We give an explicit description of the irreducible components of $W_{\ul d}(C)$ for a semistable multidegre $\ul d$. As a consequence we show that, if two semistable multidegrees of total degre $g-1$ on a curve with no rational components differ by a twister, then the respective Brill-Noether loci have isomorphic components.
\end{abstract}

\vspace{.3in}

\section{Introduction}

Let $C$ be a  projective curve of genus $g$ over an algebraically closed field $k$. If  $C$ is smooth then, for each integer $d$, the degree-$d$ Jacobian  $J_C^d$ of the curve is a projective variety parameterizing line bundles of degree $d$ on $C$.
The set $W_d(C)$ consisting of those line bundles of degree $d$ having a nonzero section can be given the structure of a  subvariety of $J_C^d$ called a  \emph{Brill-Noether variety}.
It is easy to see that  $W_d(C)$ equals the image of the \emph{Abel map of degree $d$} 
$$\begin{array}{rrcl}
\alpha_C^d\colon &C^d &\rightarrow & J_C^d\\
&(P_1,\ldots,P_d) & \mapsto & \Ocal(P_1+\ldots+P_d),
\end{array}$$
where $C^d$ is the product of $d$ copies of the curve $C$. 
In particular, each $W_d(C)$ is irreducible.

If $C$ is a singular curve, the  Jacobian varieties  $J_C^d$ are no longer projective. 
However, the Brill-Noether locus $W_d(C)$ is still a subvariety of the Jacobian  
and we can consider the problem of describing $W_d(C)$ in terms of (rational) Abel maps.
If $C$ is nodal and irreducible then again $W_d(C)$ coincides with the closure of the image of the rational Abel map of degree $d$, as shown in \cite[Thm. 1.2.1]{brannetti}. However, when $C$ is reducible, this is in general not the case.

Recently there has been a lot of interest in the Abel maps and consequently in the Brill-Noether theory for reducible curves. There are two main directions in this study. 
The first is  towards the resolution of the rational Abel map.
Few cases have been fully solved, namely the Abel map of degree $1$ for stable curves \cite{ce} and Gorenstein curves \cite{cce}, and the Abel map of degree 2 for  nodal curves \cite{cep}. Lastly, in \cite{cp},  Abel maps of any degree were constructed for stable curves of compact type.
We should also mention the relation established  in \cite{eo} between limit linear series and the
 fibers of the Abel map for a two-component curve of compact type.

The second direction is towards  the study of  Brill-Noether varieties and the relation with the image of the (possibly rational) Abel map. 
Brill-Noether varieties of any degree on a binary curve  were considered in \cite{c2} and those of degree $g-1$ on a nodal curve of genus $g$ were considered in \cite{beauville}
and  \cite{capotheta}.
The closures of the images of Abel maps in a compactification of the Jacobian were studied in \cite{brannetti} for binary curves, curves of compact type and irreducible curves.
Furthermore, 
 the Brill-Noether locus of degree $g-1$ on a compactification of the Jacobian has been shown to be an ample divisor, see  \cite{soucaris} and \cite{etheta},
and, for a stable curve, it can be regarded as  a theta divisor on a 
degeneration of an abelian varitety in the sense of Alexeev, see  \cite{alexeev}.

\subsection{Main result}

Let $C$ be a nodal reducible curve of genus $g$  and let  $W_{\ul d}(C)$ be the locus of  line bundles  of multidegree $\ul d$ on $C$ having a nontrivial section. If $\ul d$ has total degree $d$ then $W_{\ul d}(C)$ is an open and closed  subset of $W_d(C)$.
It was shown in \cite[Prop. 3.2.1]{capotheta}
 that if $\ul d$ is a  
semistable multidegree of total 
degree $g-1$ 
then the closure of the image of the rational Abel map of multidegree $\ul d$ is an irreducible component of 
$W_{\ul d}(C)$.
Moreover, by \cite[Thm. 3.1.2]{capotheta}, if $\ul d$ is stable then this is the only component and thus $W_{\ul d}(C)$ is irreducible. As we will see, this is seldom the case when $\ul d$ is not stable.

The purpose of this paper is to describe the irreducible components of the Brill-Noether variety $W_{\ul d}(C)$ in terms of images of Abel maps. 
We show (see Lemma \ref{lemaAeT}) that ``twisting" the image of an Abel map of semistable multidegrees of total degree $g-1$ also give components of $W_{\ul d}(C)$.
In general, not all components are of this form.
However, we show that every irreducible component of $W_{\ul d}(C)$
can be given as the locus of line bundles
restricting on some subcurve of $C$ to a 
bundle on the image of an Abel map of this subcurve. More precisely, for each subcurve $Z$ of $C$  let $W_{\ul d,Z}(C)$ be the closure of the locus of line bundles $L$ on $C$ of multidegree $\ul d$ such that $L|_Z(-Z\cap Z')$ lies on the image of an Abel map of $Z$.

The theorem below summarizes the description of the components of $W_{\ul d}(C)$ (cf. Theorems \ref{wdzcomponent} and \ref{thmAcomp}).

\bigskip
\noindent {\bf Theorem A.} 
{\it Let $\ul d$ be a semistable multidegree of total degree $g-1$ on a nodal curve $C$ of genus $g$. 
The irreducible components of $W_{\ul d}(C)$ are the subsets $W_{\ul d,Z}(C)$ 
where $Z$ is a connected nonempty subcurve of $C$  such that
 $\ul d_Z-\ul{\deg}(\Ocal_Z(Z\cap Z'))$ is an effective semistable multidegree on $Z$ of total degree $g_Z-1$. }

{\it Moreover, fix a component $W_{\ul d,Z}(C)$ of $W_{\ul d}(C)$. If   
$\ul d-\ul{\deg}(\Ocal_C(Z'))$ is an effective multidegree, then
 $W_{\ul d,Z}(C)$ is the
closure of the locus of line bundles of the form $L\otimes \Ocal_C(Z')$ where $L$ lies in the image of an Abel map.}

\bigskip

The  description of the Brill-Noether locus in terms of Abel maps allows us to give a sufficient conditon on the multidegrees $\ul d$ and $\ul e$ of total degree $g-1$ for which the varieties $W_{\ul d}(C)$ and $W_{\ul e}(C)$  have isomorphic components (cf. Theorem \ref{wdiso}). We believe  that, in this case, the varieties itself should be isomorphic.

\bigskip 
\noindent {\bf Theorem B.}
{\it Let $C$ be a nodal curve of genus $g$. Assume the irreducible components of $C$ have positive genus.
Let $\ul d$ and $\ul e$ be semistable multidegrees on $C$ of total degree $g-1$ such that $\ul e=\ul d+\ul{\deg}(\Ocal_C(T))$,  where $\Ocal_C(T)$ is a twister on $C$. Then $W_{\ul d}(C)$ and $ W_{\ul e}(C)$ have isomorphic components.}

\subsection{Notation and terminology}

A \emph{curve} $C$ is a connected projective reduced scheme of dimension 1 over an algebraically closed field $k$. 
A \emph{subcurve} $Z$ of  $C$ is a reduced union of irreducible components of $C$. We call  $Z':=\overline{C-Z}$ the \emph{complementary subcurve} of $Z$. 
Let $k_Z:=\#(Z\cap Z')$ and  denote by $n_Z$ the number of connected components of $Z$. 
The \emph{genus} of $Z$ is $g_Z=h^0(Z,\omega_Z)$ where  $\omega_Z$ is the \emph{dualizing sheaf} of $Z$.

Let $C_1,\ldots,C_{\gamma}$ be the irreducible components of $C$.
  A \emph{multidegree} on $C$ is a $\gamma$-uple $\ul{d}=(d_1,\ldots,d_{\gamma})\in\mathbb{Z}^{\gamma}$.
We say $\ul d$ is \emph{effective} if $d_i\geq0$ for all $i=1,\ldots,\gamma$.
The \emph{total degree} of $\ul d$ is $|\ul d|= \sum_{i=1}^{\gamma} d_i$. 
For any subcurve $Z$ of $C$ we let $\ul d_Z$ be the multidegree on $Z$ whose degree on any component of $Z$ equals that of $\ul d$. Then
$$d_Z=|\ul d_Z|=\sum_{C_i\subset Z}d_i$$
is the total degree of $\ul d_Z$.

The \emph{degree} of a line bundle $L$ on a curve $C$ is $\deg(L):=\chi(L)-\chi(\Ocal_C)$. If $C$ has components $C_1,\ldots,C_{\gamma}$, we define the \emph{multidegree} of $L$ as $\ul{\deg}(L)=(d_1,\ldots,d_{\gamma})$ where $d_i=\deg(L|_{C_i})$. Then $|\ul{\deg}(L)|=\deg(L)$.

Let $Z$ be a subcurve of a nodal curve $C$. We  
denote 
the normalization of $C$ at the points of $Z\cap Z'$ by
$$\nu_Z\colon Z\amalg Z'\rightarrow C.$$
For each multidegree $\ul d$ on $C$, the associated pullback map
$$\nu_Z^*\colon J_C^{\ul d} \rightarrow J_Z^{{\ul d}_Z}\times J_{Z'}^{{\ul d}_{Z'}}$$
maps each line bundle $L$ to the restrictions $L|_Z$ and $L|_{Z'}$. This map is a fibration with fibers  $(k^*)^{k_Z-n_Z-n_{Z'}+1}$, cf. \cite[Section 1.1]{capotheta}.

A \emph{family of curves} is a proper and flat morphism $f\colon\mathcal C\rightarrow B$ whose fibers are curves. If $b\in B$, we denote $C_b:=f^{-1}(b)$.
A \emph{smoothing} of a curve $C$ is a family $f\colon\Ccal\rightarrow B,$ where $\Ccal$ is smooth and $B$ is a smooth curve with a distinguished point $0\in B$ such that $C_b$ is smooth for $b\ne 0$ and $C_0=C$.

Let $\Ccal/B$ be a smoothing of a nodal curve $C$. 
For each subcurve $Z\subset C$ we define the \emph{twister} associated to $Z$ as
the degree-0 line bundle on $C$ given by
$$\Ocal_C(Z):=\Ocal_{\Ccal}(Z)|_C.$$
We remark that $\Ocal_C(Z)|_{Z'}=\Ocal_{Z'}(Z\cap Z')$ and  $\Ocal_C(Z)|_{Z}=\Ocal_{Z}(-Z\cap Z')$. Thus
$$\deg(\Ocal_C(Z)|_{Z'})=k_Z\quad\text{and}\quad
\deg(\Ocal_C(Z)|_{Z})=-k_Z.$$
If $Z_1,\ldots,Z_n$ are subcurves of $C$ we define the \emph{twister} associated to the formal sum $T=Z_1+\ldots+Z_n$ to be 
$$\Ocal_C(T):=\Ocal_C(Z_1)\otimes \cdots\otimes \Ocal_C(Z_n).$$

Recall that, for a nonempty subcurve $Z$ of a nodal curve $C$, we have $\omega_Z\cong \omega_C|_Z\otimes \Ocal_Z(-Z\cap Z')$ and hence
$$\deg(\omega_C|_Z)=2g_Z-2n_Z+k_Z.$$

Troughout the paper $C$ will denote a  nodal curve of genus $g$ having $\gamma$ components.

\section{Abel maps and Brill-Noether loci}

\subsection{The Brill-Noether locus $W_{\ul d}(C)$} 

Let $C$  be a nodal curve with irreducible components $C_1,\ldots,C_{\gamma}$. The  degree-$d$ Jacobian $J_C^d$  of $C$ decomposes as:
\begin{equation}\label{Jacobdecomp}
J^d_C={\underset{\ul{d}\in\Zbb^{\gamma},\, |\ul d|=d}{\bigcup}}J^{\ul{d}}_C,
\end{equation}
where $J^{\ul{d}}_C$ is the connected component of $J_C^d$ parameterizing line bundles $L$ of multidegree $\ul d$. 
Moreover, the Brill-Noether variety
$$W_d(C):=\{L\in J_C^d\ |\ h^0(C,L)\geq 1\}$$
also decomposes as a union $W_d(C)=\coprod_{|\ul d|=d} W_{\ul d}(C)$ where
$$W_{\ul d}(C):=\{L\in J_C^{\ul d}\ |\  h^0(C,L)\geq 1\}.$$

For an effective $\ul d\in\Zbb^{\gamma}$, the \emph{(natural) rational Abel map of multidegree $\ul d$  of $C$} is the rational map
$$\begin{array}{rrcl}
\alpha_C^{\ul d}\colon &  C^{\ul d} &\dashrightarrow & J_C^{\ul d}\\
& (P_1,\ldots,P_d) & \mapsto & \Ocal_C(P_1+\ldots+P_d)
\end{array}$$
where
$ C^{\ul d}= C_1^{d_1}\times \ldots \times  C_{\gamma}^{d_{\gamma}}.$
We say a line bundle on $C$ is \emph{effective} if it lies in the image of a natural Abel map.

Let $A_{\ul d}(C)$ be the closure of the image of $\alpha_C^{\ul d}$ in $J_C^{\ul d}$.  It is clear that $A_{\ul d}(C)$ is irreducible and  $\dim(A_{\ul d}(C))\leq |\ul d|$. Moreover, $A_{\ul d}(C)$ is contained in $W_{\ul d}(C)$, although equality  usually doesn't hold.

Let $\ul d$ be a multidegree of total degree  $d$. Following \cite{cce}, we say $\ul d$ is 
  \emph{(canonically) semistable}  if and only if for every non-empty, proper subcurve $Z\subsetneq C$:
\begin{equation}\label{BI}
d_Z  \geq d\,\frac{\deg \omega_C|_Z}{2g-2} - \frac{k_Z}{2},
\end{equation}
Moreover, $\ul d$ is \emph{(canonically) stable} if strict inequality holds. A degree-$d$ line bundle $L$ on $C$ is \emph{semistable} (resp. \emph{stable}) if its multidegree is.
We notice that for line bundles of total degree  $g-1$ the inequality (\ref{BI}) reduces to
\begin{equation}\label{BIg-1}
d_Z\geq g_Z-n_Z,
\end{equation}
where $g_Z$ is the genus of $Z$ and  $n_Z$ is the number of its connected components.

\subsection{Closed subsets of $W_{\ul d}(C)$}

Fix a multidegree $\ul d$ on a nodal curve  $C$. We now define some closed subsets of the Brill-Noether variety $W_{\ul d}(C)$.

Let $Z$ be a nonempty subcurve of $C$
and set $\ul e=\ul d-\ul{\deg}(\Ocal_C(Z'))$. Notice that $|\ul e|=|\ul d|$.
Assume $\ul e_Z$ is an effective multidegree of $Z$ and consider the closed irreducible subset of $J_Z^{{\ul d}_Z}$ given by $V:=A_{{\ul e}_Z}(Z)\otimes\Ocal_Z(Z\cap Z')$. Set theoretically, we have
$$V:=\{M\in J_Z^{{\ul d}_Z}\ |\ M(-Z\cap Z')\in A_{{\ul e}_Z}(Z)\}.$$
We define $W_{\ul d,Z}(C)$ to be the preimage  $(\nu_Z^*)^{-1}(V\times J_{Z'}^{{\ul d}_{Z'}})$, where $\nu_Z$ is the normalization of $C$ at the points of the intersection $Z\cap Z'$.
Set-theoretically we have
$$W_{\ul d,Z}(C)=\overline{\{L\in J_C^{\ul d}\ |\  L|_Z(-Z\cap Z')\text{ is effective}\}}.$$
Note that if $Z=C$ we have $W_{\ul d,Z}(C)=A_{\ul d}(C)$.

For a multidegree $\ul d\in\Zbb^{\gamma}$, we let $\mathcal{S}(\ul d)$ be the set of all subcurves $Z$ of $C$ satisfying the above conditions, that is, $Z$ is a nonempty subcurve of $C$ such that
$\ul d_Z-\ul{\deg}(\Ocal_Z(Z\cap Z'))$ is an effective multidegree of $Z$.
Note that if $\ul d$ is effective then $C\in \mathcal{S}(\ul d)$.

\begin{Prop}\label{decompanyd}
Let $C$ be a nodal curve of genus $g$ and let $\ul d\in\Zbb^{\gamma}$ be a multidegree on $C$. Then for each $Z\in \mathcal{S}(\ul d)$ we have   $W_{\ul d,Z}(C)\subset W_{\ul d}(C)$. Moreover,
$$W_{\ul d}(C)=\bigcup_{Z\in\mathcal{S}(\ul d)} W_{\ul d,\,Z}(C).$$
\end{Prop}
\begin{proof}
The first statement follows from computing cohomology on the  exact sequence
$$0\rightarrow L|_Z(-Z\cap Z')\rightarrow L\rightarrow L|_{Z'}\rightarrow 0.$$

Now, let $L\in W_{\ul d}(C)$ and fix a nontrivial section $s$ of $L$. 
We first note that the restriction of $s$ to each component $C_i$ of $C$ is either the zero section or an injective section of $L|_{C_i}$.
Let $Y$ be the subcurve given as the union of the components of $C$ on which $s$ is identically zero. If $Y$ is empty then $s$ is an injective section of $L$ and thus $L\in A_{\ul d}(C)$.

Now assume $Y$ nonempty and let $Z=Y'$. Consider the exact sequence
\begin{equation}\label{eqprop21}
0\rightarrow H^0(Z,L|_Z(-Z\cap Z'))
\stackrel{\psi}{\longrightarrow} H^0(C,L)
{\longrightarrow}H^0(Z',L|_{Z'}).
\end{equation}
The section $s$ vanishes on $Y=Z'$, hence it lies in the image of $\psi$. Moreover, since this section is injective on $Z$, we see that $Z\in\mathcal{S}(\ul d)$ and  $L\in W_{\ul d,Z}(C)$.
\end{proof}

\begin{Lem}\label{h01}
Let $C$ be a nodal curve and fix a multidegree $\ul d$ of $C$. If $L\in W_{\ul d}(C)$ is such that $h^0(C,L)=1$ then there exists a unique subcurve $Z$ of $C$ such that $L\in W_{\ul d,Z}(C)$.
\end{Lem}
\begin{proof}
By Proposition \ref{decompanyd}, $L\in W_{\ul d,Z}(C)$ for some $Z$. Assume $L\in W_{\ul d,Y}(C)$. 
By the proof of the proposition, we see that there exists a global section $s_Z$  (resp. $s_Y$) of $L$ that vanishes identically on $Z'$ (resp. $Y'$) and is injective on $Z$  (resp. $Y$). But since $h^0(C,L)=1$ we have $s_Y=\lambda s_Z$ for some $\lambda\in k^*$ and thus $Y=Z$.
\end{proof}

\section{Brill-Noether loci of degree $g-1$}

\subsection{Irreducible Components}

Let $C$ be a nodal curve of genus $g$. In this section we focus on line bundles of  degree  $d=g-1$. 
In this case, we have the following results by Beauville and Caporaso.

\begin{Prop}{\rm(Beauville)} \label{beauville}
Let $\ul d$ be a multidegree of total degree $|\ul d|=g-1$.
\begin{enumerate}[(i)]
\item If $\ul d$ is semistable then every irreducible component of $W_{\ul d}(C)$ has dimension $g-1$;
\item If $\ul d$ is not semistable then $W_{\ul d}(C)=J^{\ul d}_C$.
\end{enumerate}
\end{Prop}
\begin{proof}
Cf. \cite[Lemma 2.1 and Proposition 2.2]{beauville}.
\end{proof}

\begin{Prop}{\rm(Caporaso)} \label{caporaso}
Let $\ul d$ be an effective multidegree of total degree $|\ul d|=g-1$.
\begin{enumerate}[(i)]
\item If $\ul d$ is stable then $W_{\ul d}(C)=A_{\ul d}(C)$;
\item If $\ul d$ is semistable then $A_{\ul d}(C)$ has dimension $g-1$;
\item If $\ul d$ is not semistable then $A_{\ul d}(C)$ has dimension smaller than $g-1$.
\end{enumerate}
\end{Prop}
\begin{proof}
Cf. \cite[Theorem 3.1.2 and Proposition 3.2.1]{capotheta}.
\end{proof}

We see from Caporaso's result that, for a semistable multidegree, the image of the Abel map gives an irreducible component of the Brill-Noether locus. However, in general, it is not the only one. In Theorem \ref{wdzcomponent} we describe the components of $W_{\ul d}(C)$ for a semistable multidegree $\ul d$ of total degree $g-1$.

\begin{Lem}
Let $C$ be a nodal curve of genus $g$ and let $Z$ be a nonempty proper subcurve of $C$. Then 
\begin{equation}\label{genusform}
g=g_Z+g_{Z'}+k_Z+1-n_Z-n_{Z'}
\end{equation} 
where $n_Z$ (resp. $n_{Z'}$) is the number of connected components of $Z$ (resp. $Z'$).
\end{Lem}
\begin{proof}
Computing cohomology of the exact sequence
$$0\rightarrow \Ocal_C\rightarrow \Ocal_Z\oplus\Ocal_{Z'}\rightarrow  \Ocal_{Z\cap Z'}\rightarrow 0$$
we have the long exact sequence
$$0\rightarrow H^0(\Ocal_C)\rightarrow 
 H^0(\Ocal_Z)\oplus H^0(\Ocal_{Z'}) \rightarrow 
 H^0(\Ocal_{Z\cap Z'})\rightarrow $$
$$\rightarrow H^1(\Ocal_C)\rightarrow
 H^1(\Ocal_Z)\oplus H^1(\Ocal_{Z'}) \rightarrow  0.$$
Therefore
$1-n_Z-n_{Z'}+k_Z-g+g_Z+g_{Z'}=0.$
\end{proof}

\begin{Prop}\label{wdzdim}
Let $C$ be a nodal curve of genus $g$. Fix a semistable multidegree $\ul d$ on $C$ of total degree $g-1$ and a subcurve $Z\in\mathcal{S}(\ul d)$. Then $W_{\ul d,Z}(C)$ has dimension $g-1$  if and only if $Z$ is a connected subcurve of $C$ and $\ul d_Z-\ul{\deg}(\Ocal_Z(Z\cap Z'))$ is an effective semistable multidegree on $Z$ of total degree $g_Z-1$.
\end{Prop}
\begin{proof}
Let 
$\ul e:=\ul d-\ul{\deg}(\Ocal_C(Z'))$
so that 
$\ul e_Z=\ul d_Z-\ul{\deg}(\Ocal_Z(Z\cap Z')).$
Recall that $$W_{\ul d,Z}(C)=(\nu_Z^*)^{-1}(V\times J_{Z'}^{{\ul d}_{Z'}}),$$ where  $\nu_Z$ is the normalization of $C$ at the points of $Z\cap Z'$ and  $V:=A_{{\ul e}_Z}(Z)\otimes\Ocal_Z(Z\cap Z')$.
Thus $W_{\ul d,Z}(C)$ has codimension 1 in $J_C^{\ul d}$ if and only if $V$ has codimension 1 in $J_Z^{{\ul d}_Z}$.

First we notice that $V$ is isomorphic to the closure $A_{{\ul e}_Z}(Z)$ of the image of the Abel map of multidegree ${\ul e_Z}$ of $Z$ and thus
$$\dim(V)\leq |\ul e_Z|=d_Z-k_Z.$$
Hence, if $Z$ is connected and $\ul e_Z$ is 
 an effective semistable multidegree on $Z$ with total degree $g_Z-1$ then, by Proposition \ref{caporaso}, $V$ has dimension $g_Z-1$ and hence codimension $1$ in $J_Z^{{\ul d}_Z}$.

Now assume $V$ has codimension 1 in $J_Z^{\ul d_Z}$ so $\dim(V)=g_Z-1$. Then we have $d_Z\geq g_Z+k_Z-1$ which, by (\ref{genusform}), implies
$$d_{Z'}=(g-1)-d_Z\leq g_{Z'}-n_Z-n_{Z'}+1.$$
On the other hand, since  $\ul d$ is semistable we have $d_{Z'}\geq g_{Z'}-n_{Z'}$. Thus we get $n_Z= 1$ showing that $Z$ is connected.
Moreover, we see in this case that $d_{Z'}=g_{Z'}-n_{Z'}$ which implies, by (\ref{genusform}), that
$$e_Z=d_Z-k_Z=(g-1)-d_{Z'}-k_Z=g_Z-1.$$
In particular, $V$ is isomorphic to the image of an Abel map of $Z$ of total degree $g_Z-1$. So, by Proposition \ref{caporaso}, it has dimension $g_Z-1$ if and only if $\ul e_Z$ is a semistable multidegree of $Z$.
\end{proof}

\begin{Prop}\label{wdzirred}
Let $C$ be a nodal curve of genus $g$. Fix a semistable multidegree $\ul d$ on $C$ of total degree $g-1$ and a  subcurve $Z\in\mathcal{S}(\ul d)$. If $W_{\ul d,Z}(C)$ has codimension 1 in $J_C^{\ul d}$ then it is irreducible and a component of $W_{\ul d}(C)$.

Moreover, the general line bundle $L$ on such a component $W_{\ul d,Z}(C)$ satisfies $h^0(C,L)=1$.
\end{Prop}
\begin{proof}
For simplicity we denote by $f$ the restriction of $\nu_Z^*$ to $W_{\ul d,Z}(C)$ and by $Y$ the product $(A_{{\ul e}_Z}(Z)\otimes\Ocal_Z(Z\cap Z'))\times J_{Z'}^{{\ul d}_{Z'}}$.
 Recall that 
$$f\colon W_{\ul d,Z}(C) \longrightarrow Y.$$
 is a  surjective morphism, $Y$ is irreducible and the fibers of $f$ 
are all irreducible and of the same dimension. Let
$W_{\ul d,Z}(C)=W_1\cup\ldots\cup W_r$ be the decomposition of $W_{\ul d,Z}(C)$  on irreducible components.

First we show that the components $W_i$ are  unions of fibers of $f$. Indeed, since $f$ is surjective, for each $y\in Y$ we have
$f^{-1}(y)=\cup_i (W_i\cap f^{-1}(y))$, and since $f^{-1}(y)$ is irreducible, this implies that $f^{-1}(y)=W_i\cap f^{-1}(y)$ for some $i$. Hence for each $y\in Y$ we have $f^{-1}(y)\subset W_i$ for some $i$.

Now we show that there is a unique component of $W_{\ul d,Z}(C)$ that dominates $Y$. Indeed, since $Y$ is irreducible and $Y=\cup_i\overline{f(W_i)}$, then we must have $Y=\overline{f(W_i)}$ for some $i$, say $i=1$. 
Note that $Y=\overline{f(W_1\setminus\cup_{i\geq2}W_i)}$ and thus there exists an open subset $U$ of $Y$ contained in $f(W_1\setminus\cup_{i\geq2}W_i)$. 
Fix $i\geq 2$.
Then  $W_i\setminus W_1\subset f^{-1}(Y\setminus U)$. Hence $W_i$ does not dominate $Y$ and 
in particular $f(W_i)$ have dimension strictly smaller than  $\dim(Y)$. 
Since the fibers of $f$ all have the same dimension, 
this shows  that $W_i$ has dimension strictly smaller than that of $W_1$ and 
$W_1$ is an irreducible component of $W_{\ul d}(C)$.

Now, since $W_{\ul d}(C)$ is of pure dimension, by Proposition \ref{beauville}, any other component $W_i$ of $W_{\ul d,Z}(C)$ must be contained in other irreducible components of $W_{\ul d}(C)$. 
In particular,  by Proposition \ref{decompanyd}, 
it must be contained in $W_{\ul d,Y}(C)$ for some subcurve $Y$. But this would imply, by Lema \ref{h01}, that all the line bundles $L$ belonging to $W_i$ have $h^0(C,L)\geq2$. But by \cite[Lema 2.1.2]{capotheta}, the general line bundle $L$ on a fiber of $f=\nu_Z^*$ satisfies $h^0(C,L)=1$. Since the components 
$W_{\ul d,Z}(C)$ are a union of fibers of $f$ we conclude that
  $W_{\ul d,Z}(C)$ is irreducible. This shows furthermore that 
the general $L$ on $W_{\ul d,Z}(C)$ satisfies $h^0(C,L)=1$.
\end{proof}

\begin{Thm}\label{wdzcomponent}
Let $C$ be a nodal curve of genus $g$. Fix a semistable $\ul d\in\Zbb^{\gamma}$ of total degree $g-1$. The irreducible components of $W_{\ul d}(C)$ are the subsets
 $W_{\ul d,Z}(C)$ such that
 $Z$ is a connected nonempty subcurve of $C$ and $\ul d_Z-\ul{\deg}(\Ocal_Z(Z\cap Z'))$ is an effective semistable multidegree on $Z$ of total degree $g_Z-1$.

Moreover, if  $Z_1$ and $Z_2$ are distinct subcurves as above then $W_{\ul d,Z_1}(C)$ and  $W_{\ul d,Z_2}(C)$ are distinct components of $W_{\ul d}(C)$.
\end{Thm}
\begin{proof}
The first statement follows from Propositions \ref{decompanyd}, \ref{wdzdim} and \ref{wdzirred}.
The second statement follows from Proposition \ref{wdzirred} and Lemma \ref{h01}.
\end{proof}

\subsection{Twisted Abel loci}

In this section  we show that certain irreducible components of the Brill-Noether locus $W_{\ul d}(C)$ may be interpreted as closure of images of twisted Abel maps.

Fix a smoothing of a nodal curve $C$ and an effective multidegree $\ul e$ on $C$. For each twister $\Ocal_C(T)$ we define
the \emph{twisted Abel loci}
$$A_{\ul e,T}(C)=A_{\ul e}(C)\otimes \Ocal_C(T).$$
Then $A_{\ul e,T}(C)$ is a closed irreducible subset of the Jacobian of $C$ and set-theoretically we have
$$A_{\ul e,T}(C)=\overline{\{L\otimes \Ocal_C(T)\ |\ L\in J_C^{\ul e}\text{ is effective}\}}.$$
Note that  $A_{\ul e,C}(C)=A_{\ul e}(C)$.

\begin{Lem}\label{lemaAeT}
Let $C$ be a nodal curve of genus $g$ and let $\Ocal_C(T)$ be a  twister of $C$. Fix an effective multidegree $\ul e$ on $C$ and set
$\ul d=\ul e+\ul{\deg}(\Ocal_C(T))$. Then
\begin{enumerate}[(a)]
\item $A_{\ul e,T}(C)\cong A_{\ul e}(C)$;
\item $A_{\ul e,T}(C)\subset W_{\ul d}(C)$. Moreover, if both $\ul e$ and $\ul d$ are semistable of total degree $g-1$, then $A_{\ul e,T}(C)$ is an irreducible component of $W_{\ul d}(C)$;
\item if $T=Z'$ for a subcurve $Z\subset C$ then $Z\in\mathcal{S}(\ul d)$ and
$A_{\ul e,Z'}(C)\subset W_{\ul d,Z}(C).$
\end{enumerate}
\end{Lem}
\begin{proof}
Statement (a) is obvious. 
We show (b) and (c). If $T=\sum_{i=1}^{\gamma} n_iC_i$, where $C_1,\ldots,C_{\gamma}$ are the irreducible components of $C$, we let $n$ be the minimum of the $n_i$ and let $$T'=T-nC=\sum_{i=1}^{\gamma}(n_i-n)C_i.$$
 Note that $\Ocal_C(T)=\Ocal_C(T')$. Let $Z$ be the subcurve given as the union of components $C_i$ such that $n_i=n$. 

Let $M$ be a line bundle in the interior of $ A_{\ul e,T}(C)$, say $M=L\otimes\Ocal_C(T)$ where $L$ is an effective line bundle on $C$.
Recall from (\ref{eqprop21}) that $H^0(M|_Z(-Z\cap Z'))\hookrightarrow H^0(M)$.
 We will show that $H^0(M|_Z(-Z\cap Z'))\neq0$.
For this we set $Z\cap Z'=\{P_1,\ldots,P_r\}$. 
Now, if $T'=Z_1+\cdots+Z_s$ for some subcurves $Z_j$ of $C$ then 
$Z_1\cup\ldots\cup Z_s=Z'$ and hence
$$\Ocal_C(T)|_Z=\bigotimes_{j=1}^s \Ocal_C(Z_i)|_Z=\bigotimes_{j=1}^s \Ocal_Z(Z_i\cap Z)=\Ocal_Z(\sum_{i=1}^r a_iP_i)$$
for some positive integers $a_1,\ldots,a_r$.
Thus 
$$M|_Z(-Z\cap Z')=L|_Z\otimes \Ocal_Z(\sum_{i=1}^r (a_i-1)P_i)$$
is effective on $Z$ showing that $M\in W_{\ul d,Z}(C)$. Since $W_{\ul d,Z}(C)$ is closed, this implies (c) and the first part of (b). 
The second part of (b) follows from the fact  that $A_{\ul e}(C)$ is irreducible of dimension $g-1$, by Proposition \ref{caporaso}, and $A_{\ul e,T}(C)$ is isomorphic to $A_{\ul e}(C)$, by (a).
\end{proof}

\begin{Prop}\label{edssest}
Let $C$ be a nodal curve of genus $g$. Fix a semistable multidegree $\ul d$ on $C$ of total degree $g-1$ and a subcurve $Z$ of $C$. Let
 $\ul e=\ul d-\ul{\deg}(\Ocal_C(Z'))$ and assume $\ul e_Z$ is a semistable multidegree on $Z$ of total degree $g_Z-1$. Then $\ul e$ is a semistable multidegree on $C$ of total degree $g-1$.
\end{Prop}
\begin{proof}
First note that $|\ul e|=|\ul d|=g-1$ since the total degree of a twister is zero.
If $\ul e$ is not semistable then there exists a subcurve $Y$ of $C$ such that $e_Y\leq g_Y-n_Y-1$. We may assume $Y$ is connected so that $n_Y=1$. 
Let $Y_1$ (resp. $Y_2$) be the subcurve given as the union of the components of $Y$ contained in $Z$ (resp. $Z'$). Then $Y=Y_1\cup Y_2$.

First we notice that both $Y_1$ and $Y_2$ are nonempty. 
Indeed, if $Y_2$ is empty, then $Y=Y_1$ is a subcurve of $Z$. Thus, since $\ul e_Z$ is semistable of total degree $g_Z-1$, we have $e_Y\geq g_Y-n_{Y}$, contradicting the choice of $Y$. Now if $Y_1$ is empty then $Y=Y_2$ is a subcurve of $Z'$ and thus
$$e_Y=d_Y+\#(Z\cap Y)\geq g_Y-n_Y+\#(Z\cap Y)\geq g_Y-n_Y$$
again contradicting the choice of $Y$.

Then $Y_1$ ad $Y_2$ are nonempty and 
$$e_{Y_1}\geq g_{Y_1}-n_{Y_1}\quad\text{and}\quad 
e_{Y_2}\geq g_{Y_2}-n_{Y_2}+\#(Z\cap Y_2).$$
Now since $e_Y=e_{Y_1}+e_{Y_2}$, we have
$$g_{Y_1}-n_{Y_1}+ g_{Y_2}-n_{Y_2}+\#(Y_2\cap Z)\leq e_Y\leq g_Y-2.$$
By (\ref{genusform}) we have that
$$g_Y=g_{Y_1}+g_{Y_2}+\#(Y_1\cap Y_2)+1-n_{Y_1}-n_{Y_2}$$
which gives
$\#(Z\cap Y_2)+1\leq \#(Y_1\cap Y_2)$. But this contradicts the fact that $Y_1\subset Z$ thus showing that $\ul e$ is semistable.
\end{proof}

\begin{Thm}\label{thmAcomp}
Let $C$ be a nodal curve of genus $g$. Fix a semistable multidegree $\ul d$ on $C$ of total degree $g-1$. Let $W_{\ul d,Z}(C)$ be an irreducible component of $W_{\ul d}(C)$. If  $\ul e=\ul d-\ul{\deg}(\Ocal_C(Z'))$ is effective then  $W_{\ul d,Z}(C)=A_{\ul e,Z'}(C)$.
\end{Thm}
\begin{proof}
First note that, by Lemma \ref{lemaAeT}, $A_{\ul e,Z'}(C)\cong A_{\ul e}(C)$ is irreducible and
$$A_{\ul e,Z'}(C)\subset W_{\ul d,Z}(C) \subset W_{\ul d}(C).$$
By Proposition \ref{wdzdim}, $\ul e_Z$ is semistable and hence, by Proposition \ref{edssest}, so is
$\ul e$.
Thus, by Proposition \ref{caporaso}, $A_{\ul e,Z'}(C)$ has dimension $g-1$. So $A_{\ul e,Z'}(C)$ is an irreducible component of $W_{\ul d}(C)$.
\end{proof}

The hypothesis of the multidegree $\ul e$ being effective in the above theorem is a necessary one. However, this condition becomes trivial if, for instance, the components of $C$ have positive genus.

The next result shows that, if two semistable multidegrees of degree $g-1$ differ by a multidegree of a twister, then the associated Brill-Noether loci are isomorphic.

\begin{Thm}\label{wdiso}
Let $C$ be a nodal curve of genus $g$. Assume the irreducible components of $C$ have positive genus.
Fix semistable multidegrees $\ul d$ and $\ul e$ on $C$ of total degree $g-1$ such that $\ul e=\ul d+\ul{\deg}(\Ocal_C(T))$  where $\Ocal_C(T)$ is a twister on $C$. Then $W_{\ul d}(C)$ and $W_{\ul e}(C)$ have isomorphic components.
\end{Thm}
\begin{proof}
We will show that there exists a one-to-one correspondence between components of
  $W_{\ul d}(C)$ and $W_{\ul e}(C)$ maping a component of   $W_{\ul d}(C)$ isomorphically to a component of  $W_{\ul e}(C)$.
By Theorem \ref{thmAcomp}, the irreducible components of  $W_{\ul d}(C)$ and $W_{\ul e}(C)$ can be given by twisted Abel loci.
We show that for every component $A_{\ul f,Y'}(C)$ of $W_{\ul d}(C)$ there exists a component of $W_{\ul e}(C)$ isomorphic to it. 

Let  $A_{\ul f,Y'}(C)$ be an irreducible component of $W_{\ul d}(C)$. By Lemma \ref{lemaAeT} we have  $A_{\ul f,Y'}(C)\cong A_{\ul f}(C)$ and
$\ul d=\ul f+\ul{\deg}(\Ocal_C (Y'))$. 
Now note that
$$\ul e=\ul f+\ul{\deg}(\Ocal_C (Y'))+\ul{\deg}(\Ocal_C (T))=
\ul f+\ul{\deg}(\Ocal_C (T')),$$
where $T'=T+Y'$.
Then
 $A_{\ul f}(C)\otimes \Ocal_C(T')$ is a subset of $ W_{\ul e}(C)$ isomorphic to $A_{\ul f}(C)$, and thus, to  $A_{\ul f,Y'}(C)$. In particular it is irreducible of dimension $g-1$, so it is a component of $W_{\ul e}(C)$.
\end{proof}

We believe that the Brill-Noether loci 
whose multidegrees differ by a twister as in the the previous theorem are actually isomorphic. However, to show this one should determine how the irreducible components of each locus intersect. This study has not yet been carried out.

\section{Examples}

In this section we focus on two particular classes of curves, namely two-component curves and circular curves. In each case we
describe explicitly  the irreducible components of the Brill-Noether locus for a strictly semistable multidegree, that is, a semistable multidegree that is not stable.

\subsection{Two-component curves}

Let $C$ be a nodal curve having two irreducible components $C_1$ and $C_2$ of genera $g_1$ and $g_2$, respectively, meeting in $k\geq 1$ nodes. Then the genus of the curve $C$ is $g=g_1+g_2+k-1$.
 Let $\ul d=(d_1,d_2)$ be a multidegree of total degree $g-1$. Then $\ul d$ is semistable if and only if $d_i\geq g_i-1$ for $i=1,2$. Thus we have only two semistable multidegrees that are not stable, namely $\ul d=(g_1-1,g_2-1+k)$ and 
$\ul e=(g_1-1+k,g_2-1)$.
The following result shows that $W_{\ul d}(C)\cong W_{\ul e}(C)$ and can have up to two irreducible components.

\begin{Prop} Let $C$ be a two-component curve as above.
\begin{enumerate}[(a)]
\item If $g_1=g_2=0$ then $W_{\ul d}(C)=W_{\ul e}(C)=\emptyset$;

\item If $g_1=0$ and $g_2\neq0$ then 
$$W_{\ul d}(C)=W_{\ul d,C_2}(C)=A_{\ul e,C_1}(C)\quad\text{and}\quad
W_{\ul e}(C)= A_{\ul e}(C);$$

\item If $g_1\neq0$ and $g_2=0$ then 
$$W_{\ul d}(C)=A_{\ul d}(C)
\quad\text{and}\quad
W_{\ul e}(C)=W_{\ul e,C_1}(C)=A_{\ul d,C_2}(C),$$

\item If $g_1\neq0$ and $g_2\neq0$ then 
$$W_{\ul d}(C)=A_{\ul d}(C)\cup W_{\ul d,C_2}(C)=A_{\ul d}(C)\cup A_{\ul e,C_1}(C),$$
$$W_{\ul e}(C)=A_{\ul e}(C)\cup W_{\ul e,C_1}(C)=A_{\ul e}(C)\cup A_{\ul d,C_2}(C).$$

\end{enumerate}
In any case we have $W_{\ul d}(C)\cong W_{\ul e}(C)$.
\end{Prop}
\begin{proof}
Follows from Theorems \ref{wdzcomponent} and \ref{thmAcomp}. The last statement follows from Lemma \ref{lemaAeT} (a).
\end{proof}

\subsection{Circular curves}

Let $C$ be a nodal curve having irreducible components $C_1,\ldots,C_{\gamma}$. We say $C$ is a \emph{circular curve} if 
\begin{enumerate}[(a)]
\item $C_i\cap C_j$ is empty if $j\not\in\{ i-1,i,i+1\}$;

\item $C_i\cap C_{i+1}$ is a single node of $C$ for $i=1,\ldots,\gamma-1$; 

\item $C_1\cap C_{\gamma}$ is a single node of $C$.
\end{enumerate}
If $g_i$ is the genus of $C_i$ for each $i$, then the genus of $C$ is $g=1+\sum_{i=1}^{\gamma}g_i$.

\begin{Lem}\label{ssdegcirc}
Let $C$ be a circular curve of genus $g$ having $\gamma$ components. 
Let $g_i$ be the genus of the $i$-th component of $C$. 
Then $\ul d=(d_1\ldots,d_{\gamma})$ is 
 a  semistable  multidegree of total degree $g-1$ on  $C$  if and only if
\begin{enumerate}[(a)]
\item  $d_i\in\{g_i-1,g_i,g_i+1\}$; and
\item  if  $I^-$ (resp. $I^+$) is the set of $i$ such that $d_i=g_i-1$ (resp.
$d_i=g_i+1$) then either  $I^+=I^-=\emptyset$ or, up to relabeling the components of $C$, we have
$I^-=\{k_1,\ldots,k_{\ell}\}$ and
$I^+=\{j_1,\ldots,j_{\ell}\}$ where
$$1=k_1<j_1<k_2<j_2<\ldots<k_{\ell}<j_{\ell}.$$
\end{enumerate}
Moreover, $\ul d$ is stable if and only if $I^+=I^-=\emptyset$.
\end{Lem}
\begin{proof}
Assume first that $\ul d$ is semistable. Then $d_i\geq g_i-1$ by definition. 
Moreover, the subcurve $Z=C_i'$ is connected and again by semistability $d_Z\geq g_Z-1$. Hence $d_i=g-1-d_Z\leq g-g_Z$. 
By (\ref{genusform}) we have
$g=g_Z+g_i+1$,  so we deduce that  $d_i\leq g_i+1$.

Let $n^+$ be the cardinality of $I^+$ and $n^-$ be that  of $I^-$.
Then $\sum_{i=1}^{\gamma} d_i=\sum_{i=1}^{\gamma} g_i+n^+-n^-$. Since $\sum_{i=1}^{\gamma}d_i=g-1=\sum_{i=1}^{\gamma}g_i$, we must have $n^+=n^-$.
In particular, if $\ul d$ is stable then $I^-=\emptyset$ and hence also $I^+=\emptyset$.

To show (b) we note that if for some $1\leq r\leq\ell$
there exists no $j\in I^+$ such that
 $k_r<j<k_{r+1}$, then the semistability condition fails on the subcurve $Z=C_{k_r}\cup\ldots\cup C_{k_{r+1}}$, since
$d_Z=g_{k_r}+\ldots+ g_{k_{r+1}}-2=g_Z-2$.
Likewise, if 
if for some $1\leq r\leq\ell$
there exists no $k\in I^-$ such that
 $j_r<k<j_{r+1}$, then the semistability condition fails on the subcurve $Z'$ where $Z=C_{j_r}\cup\ldots\cup C_{j_{r+1}}$, since
$d_Z=g_Z+2$ which implies that $d_{Z'}=g-1-d_Z=g_{Z'}-2$, by (\ref{genusform}).

Now assume (a) and (b) hold. 
First, note that that $\ul d$ has total degree $g-1$ since
$$d=\sum_{i=1}^{\gamma}g_i+n^+-n^-=\sum_{i=1}^{\gamma}g_i=g-1.$$
Now we show $\ul d$ is semistable. 
Let $Z$ be a proper nonempty subcurve of $C$. We may assume $Z$ is connected. 
By (a) each degree $d_i$ is either $g_i-1$, $g_i$ or $g_i+1$. Moreover, 
let $n^+(Z)$ (resp. $n^-(Z)$) be the number of components $C_i$ of $Z$ such that $i\in I^+$ (resp. $i\in I^-$). 
 Then
$$d_Z=\sum_{C_i\subset Z}g_i+n^+(Z)-n^-(Z)=g_Z+n^+(Z)-n^-(Z)$$
since $g_Z=\sum_{C_i\subset Z}g_i$.
Now, since $Z$ is connected we have by (b) that 
$$-1\leq n^+(Z)-n^-(Z)\leq 1,$$
showing that $d_Z\geq g_Z-1$. Hence $\ul d$ is semistable. 

Furthermore, if $I^+=I^-=\emptyset$, then we have $n^+(Z)=n^-(Z)=0$ and hence $d_Z=g_Z$ for every connected subcurve $Z$ of $C$ and thus $\ul d$ is stable.
\end{proof}

\begin{Prop}\label{wdcirc}
Let $C$ be a circular curve of genus $g$ having $\gamma$ components. 
Let $g_i$ be the genus of the $i$-th component of $C$ and assume $g_i\geq1$ for all $i$. 
Let  $\ul d$ be a strictly semistable multidegree of total degree $g-1$. Set $I^+$ and $I^-$ as in Lemma \ref{ssdegcirc} and 
assume that   $1\in I^-$. Then $\ul d$ is effective
and the number of irreducible components of $W_{\ul d}(C)$ is 
\begin{equation}\label{nd}
n(\ul d)=1
+\sum_{1\leq r, s\leq\ell}(j_r-k_r)(k_{s+1}-j_s),
\end{equation}
where we set $k_{\ell +1}=\gamma+1$. In particular, we have $n(\ul d)\geq 1+\ell^2$.
\end{Prop}
\begin{proof}
By Theorem \ref{wdzcomponent}, the irreducible components of $W_{\ul d}(C)$ are the subsets $W_{\ul d,Z}(C)$ where
$Z$ is a  connected nonempty subcurve of $C$ such that 
$\ul e_Z:=\ul d_Z-\ul{\deg}(\Ocal_Z(Z\cap Z'))$ is an effective semistable multidegree on $Z$ of total degree $g_Z-1$.

Since $g_i\geq1$,  all semistable multidegrees on $C$ of total degree $g-1$ are effective, by Lemma \ref{ssdegcirc}. In particular $\ul d$ is effective and
 $A_{\ul d}(C)$ is a component of $W_{\ul d}(C)$ corresponding to the subcurve $Z=C$.

Now let $Z$ be a proper connected subcurve of $C$. Note that $k_Z=2$ and $g_Z=\sum_{C_i\subset Z}g_i$. Moreover, 
$$d_Z=\sum_{C_i\subset Z}d_i=g_Z+\#\{j\in I^+\ |\ C_j\in Z\} -\#\{k\in I^-\ |\ C_k\in Z\}.$$
If $Z$ does not contain  $C_j$ for any $j\in I^+$, then $d_Z\leq g_Z$ and hence $e_Z=d_Z-2\leq g_Z-2$. Therefore $W_{\ul d,Z}(C)$ is not a component of $W_{\ul d}(C)$.
 
Assume  $Z$ contains a 
$C_{j_r}$ for a unique $j_r\in I^+$.  
For simplicity we set we set $k_{\ell +1}=\gamma+1$.
If $Z$ contains either $C_{k_r}$ or $C_{k_{r+1}}$ then $d_Z\leq g_Z$ and thus $e_Z\leq g_Z-2$ and once again 
$W_{\ul d,Z}(C)$ is not a component. 
If $Z$ does not contain  $C_{k_r}$ nor $C_{k_{r+1}}$  then, since $Z$ is connected, we have $Z=C_n\cup\ldots\cup C_m$ where 
$k_r+1\leq n\leq j_r$ and $j_r\leq m\leq k_{r+1}-1$. 
In this case we have $d_Z=g_Z+1$ and thus $e_Z=g_Z-1$. Moreover, since 
$\ul e_Z=\ul d_Z-(1,0,\ldots,0,1)$ and $d_i\geq g_i$ for each $C_i\subset Z$,
 it's easy to see that $\ul e_Z$ is semistable and effective. Thus $W_{\ul d,Z}(C)$ is a component of $W_{\ul d}(C)$. The number of such components is
$$n_1:=\sum_{1\leq r\leq \ell} (j_r-k_r)(k_{r+1}-j_r),$$
where we set $k_{\ell +1}=\gamma+1$.

Now assume $Z$ contains  $C_{j_r}$ and $C_{j_s}$ for some $r<s$ and does not contain $C_{j_t}$ for any $t<r$ or $t>s$. 
Since $Z$ is connected, there are two cases: 
\begin{enumerate}[(i)]
\item $Z\supseteq C_{j_r}\cup\ldots\cup C_{j_s}$;
\item $Z\supseteq C_1\cup\ldots\cup C_{j_r}\cup C_{j_s}\ldots\cup C_{\gamma}$.
\end{enumerate}
First assume (i) $Z$ contains $C_{k_{r+1}}$. 
As in the previous case, if $Z$ contains either $C_{k_r}$ or $C_{k_{s+1}}$, then  $d_Z\leq g_Z$ and thus $e_Z\leq g_Z-2$ and 
$W_{\ul d,Z}(C)$ is not a component. 
If  $Z$ does not contain  $C_{k_r}$ and $C_{k_{s+1}}$ then, since $Z$ is connected, we must have
$Z=C_n\cup\ldots\cup C_m$ where 
$k_r+1\leq n\leq j_r$ and $j_s\leq m\leq k_{s+1}-1$. 
As before, it is easy to see that $\ul e_Z$ is semistable and effective and hence $W_{\ul d,Z}(C)$ is a component of $W_{\ul d}(C)$.
The number of such components is
$$n_2=\sum_{1\leq r<s\leq\ell}(j_r-k_r)(k_{s+1}-j_s).$$

The last case to check is (ii). As before, if $Z$ contains either 
$C_{k_s}$ or $C_{k_{r+1}}$, then  $e_Z\leq g_Z-2$ and 
$W_{\ul d,Z}(C)$ is not a component. 
If  $Z$ does not contain  $C_{k_s}$ and $C_{k_{r+1}}$ then, since $Z$ is connected, we must have
$Z=C_n\cup\ldots\cup C_{\gamma}\cup C_1\cup\ldots\cup C_m$ where 
$k_s+1\leq n\leq j_s$ and $j_r\leq m\leq k_{r+1}-1$.
As before, it is easy to see that $\ul e_Z$ is semistable and effective and hence $W_{\ul d,Z}(C)$ is a component of $W_{\ul d}(C)$. The number of such components is
$$n_3=\sum_{1\leq r<s\leq\ell}(j_s-k_s)(k_{r+1}-j_r).$$

Therefore the number of  components of $W_{\ul d}(C)$ is  $n(\ul d)=1+n_1+n_2+n_3$ and we get (\ref{nd}). To get the last statement, it's enough to note that for every $1\leq r\leq \ell$ we have $j_r-k_r\geq 1$ and $k_{r+1}-j_r\geq1$ and hence, from (\ref{nd}), we get $n(\ul d)\geq 1+\ell+\ell(\ell-1)=1+\ell^2$.
\end{proof}

\begin{Cor}
Let $C$ be a circular curve of genus $g$ having $\gamma=2\ell$ components. 
Let $g_i$ be the genus of the $i$-th component of $C$ and assume $g_i\geq1$ for all $i$. Set
$$\ul d=(g_1-1,g_2+1,g_3-1,\ldots,g_{2\ell}+1).$$ 
Then
 the number of irreducible components of $W_{\ul d}$ is
$1+\ell^2.$
\end{Cor}
\begin{proof}
Follows directly fom Proposition \ref{wdcirc}.
\end{proof}

\bigskip
\noindent J. Coelho (coelho@impa.br)
\\ Universidade Federal Fluminense
\\ Rua M\'ario Santos Braga S/N 
\\ Niter\'oi -- Rio de Janeiro -- Brazil
\newline\newline

\noindent E. Esteves (esteves@impa.br)
\\ IMPA
\\ Est. D. Castorina 110 22460-320 
\\ Rio de Janeiro -- Brazil


\begin{thebibliography}{AAA}

\bibitem[A]{alexeev}
V. Alexeev,
\emph{Compactified Jacobians and Torelli map}.
Publ. RIMS, Kyoto Univ. {\bfseries 40} (2004), 1241--1265.


\bibitem[Be]{beauville} 
A. Beauville, 
\emph{Prym varietes and the Schottky problem}. Invent. Math. {\bfseries  41} (1977), no. 2, 149--196.

\bibitem[Br]{brannetti} 
S. Brannetti, 
\emph{Compactifying the image of the Abel map}.
Preprint, available at arXiv:1009.4815 [math.AG].

\bibitem[C1]{capotheta}
L. Caporaso,
\emph{Geometry of the theta divisor of a compactified jacobian}.
J. Eur. Math. Soc. (JEMS)  {\bfseries 11}  (2009),  no. 6, 1385-–1427.

\bibitem[C2]{c2}
L. Caporaso, 
\emph{Brill-Noether theory of binary curves}.
Math. Res. Lett. vol. {\bfseries 17} (2010), no. 2,  243--262.





\bibitem[CE]{ce} 
L. Caporaso, E. Esteves, 
\emph{On Abel maps of stable curves}. 
Michigan Math. J. {\bfseries 55} (2007), no. 3, 575--607.


\bibitem[CCE]{cce}
L. Caporaso, J. Coelho, E. Esteves, 
\emph{Abel maps of Gorenstein cuves}. 
Rend. Circ. Mat. Palermo (2) vol. {\bfseries  57}  (2008),  no. 1, 33–-59.


\bibitem[CEP]{cep}
J. Coelho, E. Esteves, M. Pacini, 
\emph{Abel maps of Gorenstein cuves}. 
In preparation.


\bibitem[CP]{cp}
J. Coelho, M. Pacini, 
\emph{Abel maps for curves of compact type}. 
J. Pure Appl. Algebra  vol. {\bfseries 214}  (2010),  no. 8, 1319-–1333.

\bibitem[E]{etheta}
E. Esteves, 
\emph{Very ampleness for theta on the compactified Jacobian}.
 Math. Z. vol. {\bfseries 226} (1997), no. 2, 181--191.


\bibitem[EO]{eo}
 E. Esteves, B. Osserman, 
\emph{Abel maps and limit linear series}. 
Preprint, available at arXiv:1102.3191 [math.AG].


\bibitem[S]{soucaris}
A. Soucaris,
\emph{The ampleness of the theta divisor on the compactified Jacobian
of a proper and integral curve}.
Compositio Math.  {\bfseries 93}  (1994), 231--242.


\end{thebibliography}
\end{document}